\documentclass[preprint,12pt]{elsart}
\usepackage{graphicx}
\usepackage{amsfonts,amssymb,amsmath}
\usepackage{latexsym}
\usepackage{subfigure}
\usepackage{graphicx}
\usepackage{url}
\usepackage[T1]{fontenc}
\usepackage[utf8]{inputenc}

\newenvironment{proof}{\par \noindent {\bf Proof}.\ }{\hfill$\Box$ \par \vspace{11pt}}

\begin{document}
\begin{frontmatter}
  \title{Coloration of $K_7^-$-minor free graphs\thanksref{ANR-EGOS}}
  \author[I3M]{Boris Albar}
  \ead{boris.albar@lirmm.fr}
  \address[I3M]{I3M \& LIRMM, CNRS and Univ. Montpellier 2, Place Eugène Bataillon, 34095
    Montpellier Cedex 5, France}
  \thanks[ANR-EGOS]{This work was partially supported by the ANR grant
    EGOS 12 JS02 002 01}

  \begin{abstract}
    Hadwiger's conjecture says that every $K_t$-minor free graph is $(t - 1)$-colorable.
	This problem has been proved for $t \leq 6$ but remains open for $t \geq 7$.
	$K_7$-minor free graphs have been proved to be $8$-colorable (Albar \& Gon\c{c}alves, 2013).
	We prove here that $K_7^-$-minor free graphs are $7$-colorable, where $K_7^-$ is the graph obtained from
	$K_7$ by removing one edge.
  \end{abstract}
\end{frontmatter}

\section{Introduction}

A minor of a graph $G$ is a graph obtained from $G$ by a succession of
edge deletions, edge contractions and vertex deletions.
All graphs we consider are simple, i.e. without loops or multiple edges.

Hadwiger's conjecture says that every $t$-chromatic graph $G$
(i.e. $\chi(G) =t$) contains $K_t$ has a minor. This conjecture has
been proved for $t \leq 6$, where the case $t = 5$ is equivalent to
the Four Color Theorem by Wagner's structure theorem of $K_5$-minor
free graphs, and the case $t = 6$ has been proved by Robertson, Seymour
and Thomas~\cite{rst1}. The conjecture remains open for $t \geq 7$.

In \cite{ag1}, the author and D. Gonçalves proved that $K_7$-minor
graphs are $8$-colorable.

In \cite{kt1}, Kawarabayashi and Toft proved that any $K_7$ and
$K_{4,4}$-minor free graph is $6$-colorable by using the fact that
a $K_{4,4}$-minor free graph contains at most $4n - 8$ edges. In particular,
this implies that it contains some vertices of degree $7$. In their proof they show
that most of these vertices in a $7$-chromatic critical graph (i.e. such that every
strict minor of this graph is $6$-colorable) are contained in a
$K_5$ subgraph and use these subgraphs and the $7$-connectivity of a
$7$-chromatic critical graph to find a $K_7$-minor.

We use here similar techniques to prove the following theorem.
\begin{thm}
Every $K_7^-$-minor free graph is $7$-colorable.
\label{th:7color}
\end{thm}

\section{Proof of Theorem \ref{th:7color}}

Let $G$ be a minimal counter example to Theorem \ref{th:7color}, i.e. a minimal
$K_7^-$-minor free $8$-chromatic critical graph,

First we will prove that a lot of vertices of degree $8$ are contained in $K_5$ subgraphs
and then we will apply some techniques introduced in \cite{kt1} to conclude.

We will use the following theorem of Jakobsen to prove that $K_7^-$-minor free
graphs are $8$-degenerate.

\begin{thm}[Jakobsen, 1983, \cite{jakobsen2}]
Every graph with at least $7$ vertices and at least $\frac{9}{2} n - 12$ edges has a $K_7^-$-minor
or is a $(K_{2,2,2,2}, K_6, 4)$-cockade.
\label{th:jacokk7}
\end{thm}

We also need the following theorem of Mader.

\begin{thm}[Mader, 1968, \cite{mader1}]
Any $k$-chromatic critical graph that is not isomorphic to $K_7$ is $7$-connected for $k \geq 7$.
\label{th:maderconnected}
\end{thm}

Hence $G$ is $7$-connected, and thus is not a $(K_{2,2,2,2}, K_6, 4)$-cockade.
Thus we can deduce the following corollary of these two theorems.
\begin{cor}
$G$ has less than $\frac{9}{2} n - 12$ edges.
\label{cor:jakob7}
\end{cor}

We also need the following folklore lemma (see \cite{ag1} for a proof).
\begin{lem}[Folklore]
In a $8$-chromatic critical graph $G$, $G$ has minimum degree at least $7$ and for
any vertex $u$ of degree $7$ (resp. $8$), then the graph induced by $N(u)$ has no stable
of size $2$ (resp. $3$).
\label{lem:nostables}
\end{lem}

In particular, this lemma implies that $G$ has minimum degree at least $8$ because
if $G$ contains a vertex $u$ of degree $7$ then $N(u)$ has no stable set
of size $2$ and thus $G$ contains a $K_7$-minor, a contradiction.
We will use vertices of degree $8$ and their neighborhoods to find a $K_7^-$-minor.
The following lemma ensures the existence of such vertices.

\begin{lem}
$G$ has at least $25$ vertices of degree $8$.
\label{lem:smalldegree}
\end{lem}

\begin{proof}
By Corollary \ref{cor:jakob7}, $G$ has less than $\frac{9}{2} n - 12$ edges.
Suppose that $G$ has at most $24$ vertices of degree $8$. By Lemma \ref{lem:nostables},
$G$ has no vertices of degree strictly less than $8$, so we have that :
\[|E(G)| \geq \frac{9(n - 24) + 8 * 24}{2} = \frac{9}{2} n - 12 ,\]
a contradiction.
\end{proof}

\begin{lem}
Let $u$ be a vertex of degree $8$, then either $N(u)$ contains $K_4$ as a subgraph
or $N(u)$ contains the graph $C^{1,2}_{8}$, i.e. the circulant graph on $8$ vertices with jumps $1,2$
(see Figure \ref{fig:neighborhoods}), as a subgraph.
\label{lem:neighborhoods}
\end{lem}

\begin{figure}[h]
\centering
\includegraphics[scale=1]{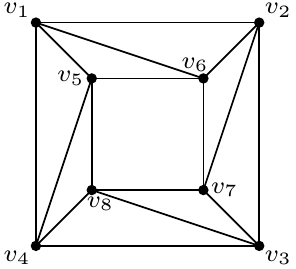}
\caption{The graph $C^{1,2}_{8}$}
\label{fig:neighborhoods}
\end{figure}

Before proving Lemma \ref{lem:neighborhoods}, let us introduce some material.
The following lemma can be immediatly deduced from the four-color theorem.
\begin{lem}
Let $x$, $y$ and $z$ be three vertices of $G$, then $G - \{x,y,z\}$ is $4$-connected and non-planar.
\label{lem:nonplanar}
\end{lem}

\begin{proof}
The first part of the lemma is obvious by the $7$-connectivity of $G$. Suppose now that there exists
$x,y,z \in V(G)$ such that $G - \{x,y,z\}$ is planar. By the Four Color Theorem, if $G - \{x,y,z\}$ is
planar then it is $4$-colorable, thus $G$ is $7$-colorable, a contradiction.
\end{proof}

We need the following definition and theorem introduced by
Robertson, Seymour and Thomas in \cite{rst1}.
\begin{defn}
Let $H$ be a graph and $T = \{v_1,v_2,v_3\}$ be a triangle. $H$ is said triangular with respect to $T$
if one of the following holds.
\begin{itemize}
\item For some $i$ $(1 \leq i \leq 3)$, $H \setminus \{v_i\}$ has maximum valency at most $2$, and either
$H \setminus \{v_i\}$ is a circuit or it has no circuit.
\item All vertices of $H$ have valency at most $3$, there is at most one 3-valent vertex $v \neq v_1,v_2,v_3$,
and $H\setminus \{v_1, v_2, v_3\}$ has no circuit.
\item All vertices of $H$ have valency at most $3$, there is a triangle $C$ with $v_1,v_2,v_3\not\in V(C)$,
every 3-valent vertex of $H$ is in $\{v_1,v_2,v_3\} \cup V(C)$, and every circuit of $H$
except these two triangles meets both $\{v_1,v_2,v_3\}$ and $V(C)$.
\end{itemize}
\end{defn}

\begin{thm}[Robertson, Seymour \& Thomas, 1993, \cite{rst1}]
Let $v_1,v_2,v_3$ be a triangle $T$ in a $4$-connected non-planar graph $H$. Let $Z$ be an induced
subgraph of $H$ such that $v_1,v_2,v_3 \in Z$ and $Z$ is not triangular with respect to $T$. Then $H$
has a $K_5$-minor $(v_1,v_2,v_3,Z_1,Z_2)$ in $H$ such that $Z_1 \cap Z, Z_2 \cap Z \neq \emptyset$.
\label{th:triangular}
\end{thm}

Let us now prove Lemma \ref{lem:neighborhoods}.

\begin{proof}
Let $u$ be a vertex of degree $8$ of $G$ and suppose that the graph induced by $N(u)$ is $K_4$-free.

\begin{claim}
$N(u)$ is $4$-connected.
\label{claim:4connected}
\end{claim}

\begin{proof}
Let $(A,B)$ be a minimal separation of $N(u)$.
Since there is no stable of size $3$ in $N(u)$ by Lemma \ref{lem:nostables}, for each pair of vertices of $v,v' \in A \setminus B$ and any
vertex in $w \in B \setminus A$, $\{v,v',w\}$ contain at least one edge. This edge cannot be $vw$ or $v'w$ because $(A,B)$ is a separation of $N(u)$.
So this must be the edge $vv'$. We deduce that both $A \setminus B$ and $B \setminus A$ are complete graphs.

If $(A,B)$ is a separation of order $1$, then either $|A \setminus B| \geq 4$ or $|B \setminus A| \geq 4$. By the previous
remark, $N(u)$ contains a $K_4$ subgraph, a contradiction.

Let suppose that $(A,B)$ is a separation of order $2$, then in this case $|A \setminus B| = |B \setminus A| = 3$.
Let $v \in A \cap B$. Since the graph induced by $N(u)$ is $K_4$-free and since $A \setminus B$ and $B \setminus A$
are triangles, there is one vertex $w \in A \setminus B$ such that $vw$ is not an edge. In the same way, there is a vertex
$w' \in B \setminus A$ such that $vw'$ is not an edge. Since $(A,B)$ is a separation of $N(u)$, then $\{v,w,w'\}$ is a
stable set of size $3$, a contradiction.

Let now suppose that $(A,B)$ is a separation of order $3$. By the previous remark, $|B \setminus A| \leq 3$ and $|A \setminus B| \leq 3$.
Since $|N(u)| = 8$ and $|A \cap B| = 3$, we can assume without loss of generality that $|A \setminus B| = 3$ and $|B \setminus A| = 2$.
Let $A \cap B = \{s_1,s_2,s_3\}$ and let $B \setminus A = \{b_1,b_2\}$. Suppose that there is a vertex $s_i$, $1 \leq i \leq 3$
and a vertex $b_j$, $1 \leq j \leq 2$, such that $s_i b_j$ is not an edge, then since $N(u)$ has no stable set of size $3$, $s_i$
is adjacent to all the vertices of the triangle $A \setminus B$ but then $N(u)$ contains a $K_4$-subgraph, a contradiction.
Thus we can assume that $b_1$ and $b_2$ are adjacent to all the vertices of $A \cap B$.

Now since $N(u)$ is $K_4$-free, $A \cap B$ is a stable set because if say $s_i s_j$ are adjacent for $1 \leq i < j \leq 3$ then
$\{b_1,b_2,s_i,s_j\}$ would be a $K_4$-subgraph, a contradiction. But then $A \cap B$ is a stable set of size $3$,
a contradiction.
\end{proof}

\begin{claim}
$N(u)$ is planar.
\label{claim:planar}
\end{claim}

\begin{proof}
Assume that $N(u)$ is non-planar. Since $N(u)$ is $4$-connected by Claim~\ref{claim:4connected},
then $N(u)$ contains a $K_5$-minor by Wagner's theorem \cite{wagner1}. Since $G$ is not
isomorphic to $N[u] = \{u\} \cup N(u)$ as it contain at least $25$ vertices, then we can find $w \in G \setminus N[u]$.
By the $7$-connectivity of $G$, there is $7$-vertex disjoint paths between $w$ and $u$.
Let denote them by $P_1,P_2,\ldots,P_7$.
We can always assume that these paths are minimal in length and thus that these paths
intersect $N(u)$ in at most one vertex.
If there is $8$ vertex-disjoint paths between $u$ and $w$, then there exists $8$ vertex-disjoint paths
between $w$ and every vertex of $N(u)$. Since $N(u)$ contains a $K_5$-minor, then $N(u)$, together with
$u$, $w$ and the $8$ paths between $N(u)$ and $w$, contains a $K_7^-$-minor, a contradiction.

So now, let $v$ be the only vertex of $N(u)$ which is not contained in any of the $7$ paths between $w$ and $N(u)$.
By Ramsey's theorem, since $N(u) \setminus \{v\}$ has $7$ vertices and no stable set of size $3$, then it contains a triangle.
Denote by $v_1$, $v_2$ and $v_3$ its vertices. Since $N(u)$ is $4$-connected,
$N(u)$ is not triangular with respect to $\{v_1,v_2,v_3\}$, and since it is $4$-connected and non-planar, then by
Theorem \ref{th:triangular}, there exists $Z_1$ and $Z_2$ such that $(v_1,v_2,v_3,Z_1,Z_2)$ is a $K_5$-minor. Since $N(u)$ does not
contain any $K_4$ subgraphs, then $|Z_1|,|Z_2| \geq 2$, so both sets $Z_1$ and $Z_2$ intersect at least one of the $7$ paths $P_i$, $1 \leq i \leq 7$.
Thus $(v_1,v_2,v_3,Z_1,Z_2, u, \underset{{1 \leq i \leq 7}}{\bigcup} (V(P_i) \setminus N[u]))$ is a $K_7^-$-minor, a contradiction.
\end{proof}

\begin{claim}
$N(u)$ does not contain any vertex of degree $6$ or greater in $G[N(u)]$.
\end{claim}

\begin{proof}
Suppose that $N(u)$ contains a vertex $v$ of degree greater or equal than $6$, then the graph induced by
$N(v)$ in $N(u)$ contains no stable set of size $3$, but then by Ramsey's theorem it contains a triangle.
Thus $N(u)$ contains a $K_4$-subgraph, a contradiction.
\end{proof}

\begin{claim}
The neighborhood of any vertex of $N(u)$ is a $4$-path, a $4$-cycle or a $5$-cycle.
\end{claim}

\begin{proof}
Let $v$ be a vertex of degree $4$ in $N(u)$, denote by $v_1$, $v_2$, $v_3$ and $v_4$ its neighbors. Suppose that its neighborhood
is not a path nor a $4$-cycle. Since the neighborhood of $v$ is triangle-free and does not contain a stable set of size $3$, it
must be two disjoint edges, say $v_1v_2$ and $v_3v_4$.
Denote by $x$, $y$ and $z$, the three vertices in $N(u) \setminus \{v,v_1,v_2,v_3,v_4\}$. $\{x,y,z\}$ is a triangle because otherwise
there is stable set of size $3$ with $v$.

Every vertex in $\{v_1,v_2,v_3,v_4\}$ sees exactly two vertices in $\{x,y,z\}$ because, either there would be a vertex of degree at most 
$3$ in $N(u)$, contradicting the $4$-connectivity of $N(u)$, or if one these vertices is adjacent to the three vertices $x$, $y$ and $z$
then $N(u)$ would contain a $K_4$ subgraph, another contradiction. Then as there are $8$ edges between $\{v_1,...,v_4\}$ and $\{x,y,z\}$,
there exists one vertex of degree $5$ in $\{x,y,z\}$ say $x$. By symmetry, we can assume that $x$ is
adjacent to $v_1$, $v_2$ and $v_3$. Now since every vertex in $\{v_1,v_2,v_3,v_4\}$ is adjacent to two vertices
in $\{x,y,z\}$, $v_1$, $v_2$ and $v_3$ are adjacent to either $y$ or $z$. But then $(v_1,v_2,v_3,v,x,\{y,z\})$
is a $K_{3,3}$-minor, contradicting Claim \ref{claim:planar}.

If $v$ is a vertex of degree $5$, then since $N(v)$ does not contain any stable set of size $3$ and any triangle,
then $N(v)$ can only be isomorphic to the cycle of length $5$.
\end{proof}

Since $N(u)$ is planar, it has at most $18$ edges by Mader's theorem, so it contains at least one vertex of degree $4$.
Let $v$ be such a vertex. Denote by $v_1$, $v_2$, $v_3$ and $v_4$ its neighbors and $x$, $y$ and $z$ its $3$ non-neighbors.
Then $C = \{v_1, v_2, v_3, v_4\}$ can induce either a $4$-path or a $4$-cycle.

Suppose that $C$ is a $4$-path.
Now the neighborhood of $v_1$ cannot induces a $4$-cycle or a $5$-cycle because this would contradict that $v$ has degree $4$.
So $v_1$ has degree $4$ and its neighborhood is a $4$-path.
By symmetry we can assume that $v_1$'s neighbors are the $4$-path $v v_2 x y$.
Moreover $\{z,v_3,v_4\}$ is a triangle since otherwise $N(u)$ would contain a stable set of size $3$ with $v_1$.
Now $y$ is adjacent to at least $1$ other vertex in $C$ because it would be of degree $3$ otherwise,
contradicting the $4$-connectivity of $N(u)$.
Planarity forces $y$ to be adjacent to $v_4$, but then $N(u)$ contains $C^{1,2}_8$.

Now suppose that $C$ is the $4$-cycle $v_1 v_2 v_3 v_4$.
Suppose that $v_1$ has degree $4$ and assume that $v_1$'s neighborhood is a $4$-path, say $v_4 v v_2 x$.
Now $\{v_3,y,z\}$ is also a triangle because otherwise there is a stable set of size $3$ with $v_1$.
As $y$ and $z$ have degree at least $4$ in $N(u)$ and as $y,z \not\in N(v) \cup N(v_1)$ then
$y$ and $z$ are both adjacent to at least one vertex in $\{v_2, v_4\}$.
Moreover $y$ and $z$ cannot be both adjacent to the same vertex because otherwise there would
be a $K_4$-subgraph with $v_3$.
So either $y$ is adjacent to $v_2$ and the $z$ is adjacent to $v_4$ either $z$ is
adjacent to $v_2$ and the $y$ is adjacent to $v_4$. In both cases, after removing the edge
$v_2 v_3$, the graph is isomorphic to $C^{1,2}_8$. Note that the same argument applies
when $v_1$'s neighborhood is a $4$-cycle by also removing the edge $v_4 x$ at the end.

Suppose now that $v_1$ has degree $5$ so we can assume that its neighborhood is the $5$-cycle $v_4 v v_2 x y$ .
Then $z$ is adjacent to $v_3$ because otherwise $\{v_1,v_3,z\}$ is a stable set of size $3$.
Since $z$ has degree at least $4$ in $N(u)$ it is also adjacent to at least one vertex in the set $\{v_2,v_4\}$. But if
$z$ is adjacent to $v_2$ then after removing the edge $v_1 v_2$, the graph is isomorphic to $C^{1,2}_8$, and if $z$ is adjacent
to $v_4$ then after removing the edge $v_1 v_4$, the graph is isomorphic to $C^{1,2}_8$.
\end{proof}

\begin{lem}
Let $u$ and $u'$ be two degree $8$ vertices of $G$ such that $N(u)$ and $N(u')$ contain the graph $C^{1,2}_{8}$ as a subgraph,
then $u$ and $u'$ are not adjacents.
\label{lem:noadjacent}
\end{lem}

\begin{proof}
Let suppose that $u$ and $u'$ are adjacent. Since every vertex of $N(u)$ has degree at least $4$ and $u' \in N(u)$
by hypothesis, denote by $v_1$, $v_2$, $v_3$ and $v_4$ the four neighbors of $u'$ in the subgraph $C^{1,2}_8$ of $N(u)$.
These four vertices induce a path in this subgraph, say $v_1v_2v_3v_4$. Let denote by $w_1$, $w_2$ and $w_3$ the vertices of
$N(u) \setminus \{u',v_1,v_2,v_3,v_4\}$ in a way that $w_1$ is the only vertex adjacent to both $v_1$ and $v_2$ and $w_2$ is
the one adjacent to $v_3$ and $v_4$. Now consider $H = G \setminus \{u,w_1,w_2\}$. $H$ is $4$-connected and non-planar
by Lemma \ref{lem:nonplanar}. Let $Z = \{w_3, u',v_1,v_2,v_3,v_4\}$, then $Z$ is not triangular with respect to $\{u',v_1,v_2\}$
because $u'$ has degree $4$ in $Z$. Thus by Theorem \ref{th:triangular}, there exists $Z_1$ and $Z_2$ such that $(u',v_1,v_2,Z_1,Z_2)$
is a $K_5$-minor in $H$ and such that $Z_1 \cap Z,Z_2 \cap Z \neq \emptyset$. But then $(u,u',v_1,v_2,Z_1,Z_2,\{w_1,w_2\})$ is
a $K_7^-$-minor in $G$, a contradiction.
\end{proof}

The following lemma is the key to prove that a lot of degree $8$ vertices are contained in a $K_5$.

\begin{lem}
Let $u$ and $u'$ be two vertices of degree $8$ such that $N(u)$ and $N(u')$ contain the graph $C^{1,2}_{8}$ as a subgraph
and $|N(u) \cup N(u')| \geq 9$, then $G$ contains a $K_7^-$-minor.
\label{lem:nodifferent}
\end{lem}

\begin{proof}
By Lemma \ref{lem:noadjacent}, we can assume that $u$ and $u'$ are not adjacent.
Denote by $v_1,\ldots,v_8$ the vertices of $N(u)$ as shown in Figure \ref{fig:neighborhoods}.
Since $G$ is $7$-connected, there is at least $7$ internally disjoint paths between $u$ and $u'$ that induce $7$ disjoint paths
between $N(u)$ and $N(u')$. Note that theses paths can be of length $0$ if the two neighborhoods intersect.
By contracting the non-zero length paths, we obtain a graph with $|N(u) \cup N(u')| = 9$.
From now on, we consider only this new graph $G'$. By construction of $G'$,
$N(u)$ still contain a $C^{1,2}_8$-subgraph.

By symmetry of $C^{1,2}_8$, we can assume that $v_1$ is the only neighbor of $u$
which is not a neighbor of $u'$. In particular, we have that $v_i \in N(u')$ for all $i \ge 2$.
But then $(u, \{u',v_5\}, v_2, v_3, v_6, v_7, \{v_1,v_4,v_8\})$ is a $K_7^-$-minor (only $v_3$ and $v_6$
are not adjacent) of $G'$ and thus a $K_7^-$-minor of $G$, a contradiction.
\end{proof}

\begin{claim}
Let $u$ and $u'$ be two vertices of degree $8$ such that $N(u)$ and $N(u')$ contain the graph $C^{1,2}_{8}$ as a subgraph,
then $N(u) \neq N(u')$.
\label{claim:nosame}
\end{claim}

\begin{proof}
Suppose that there exists two vertices $u$ and $u'$ of degree $8$ such that
$N(u) = N(u')$. Then we can create a $K_7^-$-minor in $G$ by using the same argument
as in the proof of Lemma \ref{lem:nodifferent}, a contradiction.
\end{proof}

\begin{claim}
At most one vertex of degree $8$ have a neighborhood containing the graphs $C^{1,2}_{8}$
as a subgraph.
\label{claim:notmuch}
\end{claim}

\begin{proof}
Suppose that there exists two vertices of degree $8$ such that
their neighborhood contains the graph $C^{1,2}_{8}$ as a subgraph.
By Claim \ref{claim:nosame}, these two vertices have a different neighborhood.
By Lemma \ref{lem:nodifferent}, this imply that there is a $K_7^-$-minor in $G$, a contradiction.
\end{proof}

\begin{lem}
There is at least $5$ different $K_5$ in $G$.
\label{lem:lotofk5}
\end{lem}

\begin{proof}
By Lemma \ref{lem:smalldegree}, there is at least $25$ vertices of degree $8$ and by
Lemma \ref{claim:notmuch}, there is at most one vertices of degree $8$ containing the graph $C^{1,2}_{8}$
as a subgraph of their neighborhood. By Lemma \ref{lem:neighborhoods}, this imply that there is at least $24$ vertices
of degree $8$ that contains a $K_4$ in their neighbourhood. As every $K_5$-subgraph can contain at most $5$
vertices of degree $8$, this finally imply that there is at least $\lceil \frac{24}{5} \rceil = 5$ different $K_5$-subgraph in $G$.
\end{proof}

The following lemma is the last key to the proof. It uses techniques introduced by Kawarabayashi
and Toft \cite{kt1}.
\begin{lem}
There is $3$ different copies of $K_5$ $L_1$, $L_2$ and $L_3$ such that $|L_1 \cup L_2 \cup L_3| \geq 12$.
\label{lem:diffk5}
\end{lem}

\begin{proof}
Assume by contradiction that no three copies of $K_5$, denoted $L_i$, $L_j$ and $L_k$, are such that
$|L_i \cup L_j \cup L_k| \geq 12$.

The next claim follows easily from the $7$-connectivity of $G$.
\begin{claim}
$G$ does not contain a $K_6^-$ subgraph.
\label{claim:nok6minus}
\end{claim}

\begin{proof}
Suppose that $G$ contains a $K_6^-$ subgraph. Since $G$ is not isomorphic to $K_6^-$, there exists
a vertex that is not contained in this $K_6^-$ subgraph. Since $G$ is $7$-connected, by Menger's theorem
there are $7$ vertex-disjoint paths between $x$ and the vertices of the $K_6^-$ subgraph. This induces a
$K_7^-$-minor, a contradiction.
\end{proof}

\begin{claim}
Two different $K_5$ intersects on at most $2$ vertices.
\label{claim:nomorethantwo}
\end{claim}

\begin{proof}
Let $L_1$ and $L_2$ be two copies of $K_5$ of $G$ and suppose that they intersect on $4$ vertices, then $G$ contains a $K_6^-$ as a subgraph,
contradicting Claim \ref{claim:nok6minus}. If they intersect on $3$ vertices, then denote
by $S$ the set of vertices in $L_1 \cap L_2$ and by $H$ the set of vertices of $L_1 \Delta L_2$.
By Lemma \ref{lem:nonplanar}, $G \setminus S$ is $4$-connected and non-planar so by (2.6) of \cite{rst1} there is
a $K_4$-minor rooted in $H$ and a $K_7$-minor in $G$, a contradiction.
\end{proof}

\begin{claim}
No two $K_5$ are disjoints.
\label{claim:nodisjoint}
\end{claim}

%

\begin{proof}
Assume that $L_1$ and $L_2$ are two disjoint copies of $K_5$. For any copy of $K_5$ $L_3$, since
two copies of $K_5$ cannot intersect on $4$ vertices and since $|L_1 \cup L_2 \cup L_3| < 12$,
$|L_3 \cap L_1| \geq 2$ and $|L_3 \cap L_2| \geq 2$.
By Claim \ref{claim:nomorethantwo}, $|L_3 \cap L_1| = 2$ and $|L_3 \cap L_2| = 2$.
Let $L_3 \cap L_1 = \{a,b\}$ and $L_3 \cap L_2 = \{c,d\}$.

Now $G \setminus \{a,b,c,d\}$ is $3$-connected so by Menger's theorem there are $3$ vertex disjoint
paths $P_1$, $P_2$ and $P_3$ between $L_1 \setminus \{a,b\}$ and $L_2 \setminus \{c,d\}$ but then
$(a,b,c,d,V(P_1), V(P_2), V(P_3))$ is a $K_7$-minor, a contradiction.
\end{proof}

\begin{claim}
No two $K_5$ intersect on exactly one vertex.
\label{claim:noone}
\end{claim}

\begin{proof}
Assume that $L_1 \cap L_2 = \{x\}$. Let $L_3$ be a copy of $K_5$ different from $L_1$ and $L_2$.
By Claim \ref{claim:nodisjoint}, $L_3$ intersects both $L_1$ and $L_2$.

Suppose that $x \in L_3$. Since $|L_1 \cup L_2 \cup L_3| < 12$,
$|L_1 \cup L_3| = |L_2 \cup L_3| = 2$. Let $y \in (L_1 \cap L_3) \setminus \{x\}$.
$G \setminus \{x,y\}$ is $5$-connected and non-planar by Lemma \ref{lem:nonplanar}.
Let $Z = (L_1 \cup L_2 \cup L_3) \setminus \{x,y\}$. Denote $T = L_2 \setminus \{x,y\} = \{v_1,v_2,v_3\}$.
$Z$ is not triangular with respect to $T$, hence there exists $Z_1,Z_2$ such that $(v_1,v_2,v_3,Z_1,Z_2)$ is
a $K_5$-minor in $G \setminus \{x,y\}$ and such that $Z_1 \cap Z, Z_2 \cap Z \neq \emptyset$.
Moreover we can assume without loss of generality that $y$ is adjacent to $Z_1$. Thus $(v_1,v_2,v_3,Z_1,Z_2,x,y)$ is a $K_7^-$-minor in $G$
(only $y$ and $Z_2$ may not be adjacent), a contradiction.

Suppose now that $x \not\in L_3$. Since $|L_1 \cup L_2| = 9$ and $|L_1 \cap L_2 \cap L_3| < 12$, by Claim
\ref{claim:nomorethantwo}, we can assume that $|L_3 \cap L_1| = 2$. Let us denote $L_3 \cap L_1 = \{a,b\}$.

If $|L_3 \cap L_2| = 1$, let $\{c\} = L_3 \cap L_2$. Now $G \setminus \{a,b,c,x\}$
is $3$-connected. So by Menger's theorem, there are $3$ vertex disjoint paths
$P_1$, $P_2$ and $P_3$, between $(L_1 \cup L_3) \setminus \{a,b,c,x\}$ and
$L_2 \setminus \{c,x\}$. Hence $(a,b,c,x,V(P_1), V(P_2), V(P_3))$ is a $K_7$-minor,
a contradiction.

If $|L_3 \cap L_2| = 2$, let $\{c,d\} = L_3 \cap L_2$.
$G \setminus \{a,b,c,d,x\}$ is $2$-connected, so by Menger's theorem,
there are $2$ vertex disjoint paths $P_1$ and $P_2$
between $(L_1 \cup L_3) \setminus \{a,b,c,d,x\}$ and $L_2 \setminus \{c,d,x\}$.
But $(a,b,c,d,x,V(P_1), V(P_2))$ is a $K_7$-minor, a contradiction.
\end{proof}

\begin{claim}
No two $K_5$ intersect on exactly two vertices.
\label{claim:notwo}
\end{claim}

\begin{proof}
Assume that $L_1 \cap L_2 = \{x,y\}$. Let $L_3$ be a $K_5$ different from $L_1$ and $L_2$.
By Claims \ref{claim:nomorethantwo}, \ref{claim:nodisjoint} and \ref{claim:noone}, $L_3$ intersects each $L_1$ and $L_2$
on two vertices.

Suppose that $L_1 \cap L_2 \cap L_3 = \emptyset$ and let $L_1 \cap L_3 = \{u,v\}$
and $L_2 \cap L_3 = \{z,t\}$. Then $\{u,v,x,y,z,t\}$ is a $K_6$-subgraph, a contradiction with Claim \ref{claim:nok6minus}.

Suppose that $L_1 \cap L_2 \cap L_3 = \{x\}$ and let $(L_1 \cap L_2) \setminus \{x\} = \{y\}$,
$(L_1 \cap L_3) \setminus \{x\} = \{z\}$, $(L_2 \cap L_3) \setminus \{x\} = \{t\}$.
Now $G \setminus \{x,t\}$ is $5$-connected and non-planar. Let $Z = (L_1 \cup L_2) \setminus \{x,t\}$ and
let $T = \{v_1,v_2,v_3\} = L_2 \setminus \{x,t\}$. $Z$ is not triangular with respect to $T$, so there exists
$Z_1$ and $Z_2$ such that $(v_1,v_2,v_3,Z_1,Z_2)$ is a $K_5$-minor in $G \setminus \{x,t\}$.
Without loss of generality, we can assume that $z \in Z_1$ but then $(v_1,v_2,v_3,Z_1,Z_2,x,t)$ is a $K_7^-$-minor in $G$,
a contradiction.

Finally, suppose that $L_1 \cap L_2 \cap L_3 = \{x,y\}$, $G \setminus \{x,y\}$ is $5$-connected and non-planar.
Let $Z = (L_1 \cup L_2 \cup L_3) \setminus \{x,y\}$ and $T = \{v_1,v_2,v_3\} = L_1 \setminus \{x,y\}$.
$Z$ is not triangular with respect to $T$ so there exists $Z_1$ and $Z_2$ such that $(v_1,v_2,v_3,Z_1,Z_2)$
is a $K_5$-minor in $G \setminus \{x,y\}$, but then $(v_1,v_2,v_3,Z_1,Z_2,x,y)$ is a $K_7$-minor in $G$,
a contradiction.
\end{proof}

Claims \ref{claim:nodisjoint}, \ref{claim:noone} and \ref{claim:notwo} together with Claim \ref{claim:nomorethantwo}
conclude the proof of the lemma.
\end{proof}

We conclude the proof of Theorem \ref{th:7color} by using the following theorem due to Kawarabayashi and Toft \cite{kt1}.

\begin{thm}[Kawarabayashi \& Toft, 2005, \cite{kt1}]
Let $G$ be a $7$-connected graph with at least $19$ vertices. Suppose that $G$ contains
three $K_5$, say $L_1$, $L_2$ and $L_3$, such that $|L_1 \cup L_2 \cup L_3| \geq 12$, then
$G$ contains a $K_7$-minor.
\end{thm}

Applying this theorem to the three $K_5$ given by Lemma \ref{lem:diffk5} gives us a contradiction.

\section{Conclusion}

We have seen that $K_7^-$-minor free graphs are $7$-colorable. The techniques used here are not sufficient to prove
that $K_7$-minor free graphs are $7$-colorable because we then have to deal with "sparse" neighborhoods
of degree $8$ and $9$ vertices. However, since $6$-connected $K_8^-$-minor free graphs are $10$-degenerated \cite{song1},
we wonder whether similar techniques can be extended to prove that $K_8^-$-minor free graphs are $9$-colorable.
Currently the best bound for $K_8^-$-minor free graphs is given by the fact that $K_8$-minor free graphs are $10$-colorable \cite{ag1}.

\bibliographystyle{plain}
\nocite{*}
\bibliography{biblio}
\end{document}